\documentclass[11pt]{amsart}

\usepackage[utf8]{inputenc}

\usepackage{graphbox} 
\usepackage{mathtools, amsthm, amsfonts, amssymb, amsmath}
\usepackage{accents}
\usepackage{microtype}
\usepackage{enumerate}

\usepackage[normalem]{ulem}

\usepackage{multirow}
\usepackage{array}
\usepackage{longtable}
\usepackage{thmtools}
\usepackage{xcolor}
\usepackage{tikz}
\usepackage{tikz-cd}
\usetikzlibrary{patterns,external}
\usetikzlibrary{calc}
\usetikzlibrary{shapes.multipart}

\usepackage[numbers,sort]{natbib}

\usepackage[linesnumbered,commentsnumbered,ruled,vlined]{algorithm2e}

\usepackage[colorlinks]{hyperref}
\hypersetup{
  linkcolor=[rgb]{0.3,0.3,0.6},
  citecolor=[rgb]{0.2, 0.6, 0.2},
  urlcolor=[rgb]{0.6, 0.2, 0.2}
}

\usepackage{listings}
\definecolor{maroon}{RGB}{133, 5, 63}
\definecolor{teal}{RGB}{0, 128, 96}
\definecolor{forestgreen}{RGB}{34, 139, 34}
\lstdefinelanguage{code}{
basicstyle=\small\ttfamily,
alsoletter=",
classoffset=1,
keywords={solve, differentiate, subs, sub, real_solutions, det, sum, max, count, conditional_count, map, filter, map_filter, union, zip, product, is_finite, is_successful, is_nonsingular, track, is_real, is_success, first, evaluate, flatten, bitmask_filter, accumulate, polynomial_interpolants, reverse, coefficients, vcat, solution_from_necklace, length, prod, compress, total_degree_start_solutions, degrees, variables, iterate, struct, collect, convert, stretched_cube, stretched_cubes, weight_vector, weight_vectors, perm_to_segments, perm_to_mixedcell, mixed_cell_iterator, rand_approx_unit, norm, fixed, polyhedral_system, reduce, findall, eachrow},
keywordstyle={\color{teal}},
classoffset=2,
morekeywords={@var, @time, for, end, if, while, else, begin},
keywordstyle={\color{maroon}},
classoffset=3,
morekeywords={using, function, return, const},
keywordstyle={\color{blue}},
classoffset=4,
morekeywords={julia, >},
keywordstyle={\color{forestgreen}},
xleftmargin=1.5cm,
xrightmargin=1em,
columns=fullflexible,
keepspaces=true,
}

\makeindex

\definecolor{fondo}{rgb}{0.898,0.996,0.898}

\pgfkeys{/tikz/.cd,
  K/.store in=\K,
  K=1   
   }

\author[P. Breiding]{Paul Breiding}
\address[P. Breiding]{Department of Mathematics, University of Osnabr\"uck, Germany (ORCID: 0000-0003-3747-9185)}
\email{pbreiding@uni-osnabrueck.de}

\author[T. Brysiewicz]{Taylor Brysiewicz}
\address[T. Brysiewicz]{Department of Mathematics, University of Western Ontario, London, Canada (ORCID: 0000-0003-4272-5934)}
\email{tbrysiew@uwo.ca}

\author[H. Friedman]{Hannah Friedman}
\address[H. Friedman]{Department of Mathematics, University of California, Berkeley, USA (ORCID: 0009-0007-7831-2636)}
\email{hannahfriedman@berkeley.edu}

\newcommand{\R}{{\mathbb{R}}}

\newcommand{\x}{\textbf{x}}

\newcommand{\mycomment}[1]{}

\usepackage[all,cmtip]{xy}

\renewcommand{\tilde}[1]{\widetilde{#1}}

\newcommand{\p}{\textbf{p}}

\theoremstyle{definition}
\newtheorem{theorem}{Theorem}[section]
\newtheorem{definition}[theorem]{Definition}
\newtheorem{lemma}[theorem]{Lemma}

\newtheorem{proposition}[theorem]{Proposition}

\newcommand{\mydef}[1]{{\color{blue}{#1}}}

\DeclareRobustCommand{\code}[1]{\lstinline!#1!}
\usepackage{tikz}
\usetikzlibrary{calc}
\tikzstyle{vertex}=[circle, draw, inner sep=0pt, minimum size=6pt, fill=black]

\title{Homotopy Iterators}

\date{}

\usepackage[margin=1in]{geometry}

\begin{document}

\begin{abstract}
We introduce the concept of homotopy iterators for performing polynomial homotopy continuation tasks in a memory efficient manner. The main idea is to push forward an iterator for the start solutions of a homotopy via the function which tracks them along the homotopy. Doing so produces a representation of the target solutions, bypassing the need to hold all solutions in memory. We discuss several applications of this datatype ranging from solution counting to data compression.
\end{abstract}

\maketitle

\bigskip
\section{Introduction}
Algorithmic advances in computer algebra and numerical algebraic geometry 
have historically focused on making polynomial-system-solving \emph{faster}. 
This mission has been incredibly successful, culminating in recent developments of highly efficient polynomial homotopy continuation solvers  \cite{HOM4PS,HC, NAG4M2, PHC, Bertini}. As a result, the bottleneck of state-of-the-art software involves issues of \emph{memory}: personal computers lack the memory required to store billions of solutions. 

One observation underlying this article is that many computational tasks which require polynomial system solving do not, in fact, need all solutions to be stored simultaneously. Such tasks include
\begin{itemize}
 \item counting the number of solutions satisfying some condition,\smallskip
 \item accumulating a function of the solutions,\smallskip
 \item determining  a monodromy permutation,\smallskip
 \item finding the solution which optimizes some objective function, and\smallskip
 \item finding and sharing a particular subset of solutions. 
\end{itemize}
Therefore, building a data structure for solutions of systems of polynomial equations that allows the user to solve the above problems without storing all solutions is imperative. Such a data structure is given by an \mydef{iterator}. An iterator represents a list  via the ability to iterate through its elements \cite{Watt}. Familiar operations on lists, like pushing the list forward under a function or taking a subset satisfying certain conditions, may be performed instantaneously on the level of iterators. For example, one may push forward the start solutions of a homotopy with respect to the function which performs path-tracking on those start solutions. The output, which we call a \mydef{homotopy iterator}, represents the result of this process. The formal definition is given in \autoref{def_hom_iter}.

One contribution of this work is the user-friendly implementation of homotopy iterators in \code{HomotopyContinuation.jl} \cite{HC} v.2.15.1 (or higher) in \code{Julia}~\cite{julia}. Here is a small code example. 

\begin{lstlisting}[language=code]
using HomotopyContinuation
@var x y 
F = [x^2 + y - 1; x^2 + y^2 - 4]
I = solve(F; iterator_only = true)
\end{lstlisting}
The system \code{F} in this example has four zeros.  Without the flag  \code{iterator_only = true}, the \code{solve} function computes and return the list of the four zeros of \code{F}. Instead, setting the flag instantly returns an iterator \code{I} for the four solutions. Queries involving this solution set can be called on the level of iterators. For instance, one may ask  if there is a real zero: \code{Iterators.any(is_real, I)}.
This command iterates through~\code{I} and stops if one of the solutions is real. Throughout this process, the computer only ever holds a single zero of \code{F} in memory. We give more sophisticated examples in~\autoref{sec:implementation} of how to use iterator specific functions to reduce the memory and time requirements of some common tasks involving zeros of polynomial systems. 

Often, one desires a subset of the results represented by \code{I}, like those results which are finite, real, or nonsingular. Representing these subsets is easy using iterators: one simply filters the results based on which evaluate to \code{true} under a boolean-valued function $f$. For applications such as data compression, one may want to save the values of $f$ in a \mydef{bit-vector}, or \mydef{bitmask}, which may be transmitted at low memory cost. This is illustrated in \autoref{fig:homotopypic}. In our running example, $f$ is the function that returns \code{true} if and only if a solution is real. The two methods look as follows. 

\begin{lstlisting}[language=code]
J1 = Iterators.filter(is_real,I)
J2 = bitmask_filter(is_real, I)
\end{lstlisting}
The iterator \code{J1} is constructed instantly with no true computation involved, whereas \code{J2} is obtained by iterating through the list represented by \code{I} and attaching a bitmask \code{B} that encodes the desired subset.  Although \code{J1} and \code{J2} represent the same list, subsequent computations on \code{J2} are faster since the information about $f$ has been cached.

The expressivity of homotopy iterators goes well-beyond the ability to filter. Functionally, homotopy iterators may be used in \emph{any} setting in place of the solution set itself. From a user perspective, the difference is that homotopy iterators postpone all computation until the last moment, and do so in a way which traverses the solution set so that simultaneous storage of all solutions is \emph{never} required. 
This observation, that solution sets of polynomial systems can, and should, be represented via iterators, was first made in \cite{monodromyCoordinates}. There, an iterator is built for a solution set based on its construction via monodromy solving  \cite{Duff}. The present work builds on this idea, offering a much simpler approach when the solution set is constructed via homotopy.

\begin{figure}
\includegraphics[scale=0.5]{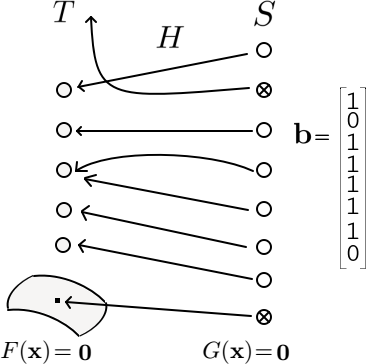}
\smallskip 

\caption{A visual depiction of a homotopy from a start system $ G(\textbf{x}) = 0$ with start solutions $S$ (represented as an iterator) to a target system $F(\textbf{x})=0$. The start solutions which correspond to finite isolated target solutions are encoded in the bit-vector $\textbf{b}$.} 
\label{fig:homotopypic}
\end{figure}

We describe homotopy continuation in \autoref{sec:polynomial_homotopy_continuation}. In \autoref{sec:iterators} we give the minimal necessary background regarding iterators and define how path-tracking can be used to push forward start solution iterators to target solution iterators. 
\autoref{sec:start_system_iterators} outlines how to produce start solution iterators in three important settings: total degree homotopies, polyhedral homotopies, and parameter homotopies. Finally, in \autoref{sec:implementation_applications_examples}, we describe our implementation in \code{HomotopyContinuation.jl} and showcase the power of our lazy-evaluation approach toward homotopy continuation through applications and examples.

\subsection*{Acknowledgements} PB is supported by Deutsche Forschungsgemeinschaft (DFG) -- Projektnr.\ 445466444. TB is supported by NSERC Discovery Grant RGPIN-2023-03551.

\bigskip
\section{Polynomial homotopy continuation}\label{sec:polynomial_homotopy_continuation}
\mydef{Polynomial homotopy continuation (PHC)} \cite{NAGbook} is a method designed to find all isolated complex solutions to a  polynomial system of the form
\begin{equation}
\label{eq:system}
F(x_1,\ldots,x_n) = F(\x) = \begin{bmatrix} f_1(\x) \\ f_2(\x) \\ \vdots \\ f_n(\x) \end{bmatrix} = \textbf{0}, \quad \quad f_1,\ldots,f_n \in \mathbb{C}[x_1,\ldots,x_n].
\end{equation}
This system is \mydef{square} since the number of equations equals the number of variables. PHC computes the finitely many isolated solutions to \eqref{eq:system} by constructing a \mydef{start system}  $G(\x)=\textbf{0}$ which ``looks like'' \eqref{eq:system} but is easy to solve. After finding an appropriate start system, PHC computes the isolated solutions to $F(\x)=\textbf{0}$ by connecting them to the solutions of $G(\x)=\textbf{0}$ via a \mydef{homotopy} $H(\x;t)$ such that $H(\x, 1) = G(\x)$ and  $H(\x, 0) = F(\x)$. 

For instance, the straight line from $F$ to $G$ gives the \mydef{straight-line homotopy}
 \begin{equation}
 \label{eq:homotopy}
 H(\x;t) = (1-t)F(\x) + \gamma \cdot t G(\x), \quad \quad \gamma \in \mathbb{C}^\times = \mathbb{C}-\{0\}.
 \end{equation}
   Another example is when $F=F(\x;\mathbf p)$ depends on $k$ parameters $\mathbf p\in\mathbb C^k$. Then, we can set up a \mydef{parameter homotopy} via 
  \begin{equation}
 \label{eq:homotopy2}
 H(\x;t) =  F(\x; (1-t)\mathbf p + t\mathbf q), \quad \quad \mathbf p, \mathbf q\in\mathbb C^k.
 \end{equation}
 We assume that $G(\x)$ is chosen appropriately, so that $t=1$ is a \textit{generic} value in the sense of the \textit{parameter continuation theorem} (see \cite{BOROVIK2025102373, MS1989}). All examples in this paper satisfy this condition which guarantees that the zeros of $H(\x;t)$ for $t=(0,1]$ are comprised of some number of smooth disjoint \mydef{homotopy paths}
 \[
 Z_t=\{z_1(t),\ldots,z_d(t)\}
 \]
 which connect the \mydef{start solutions} $S=Z_1$ of $G(\x) = H(\x;1) = \textbf{0}$ to the \mydef{target solutions} $T=Z_0$ of $F(\x) = H(\x;0) = \textbf{0}$. Implicitly, we have a map
$$
 f_H:S \to T,\quad 
 z_i(1) \mapsto z_i(0),
$$
 which may be evaluated numerically using standard path-tracking methods.

Homotopy paths may exhibit the following behavior at $t =0$:
 \begin{itemize}
 \item[]   {(a) Divergence:} A coordinate of a path $z_i(t)$ approaches $\infty$ as $t\to 0$.  \smallskip
 \item[] {(b) Path Collision:} Two paths collide, so that $z_i(0) = z_j(0)$.\smallskip
 \item[]   {(c) Positive Dimensionality:} A point $z_i(0)$  is not an isolated solution to $H(\x;0)=\textbf{0}$.\smallskip
 \end{itemize} 
\autoref{fig:homotopypic} illustrates each of these so-called ``endgame behaviors.'' The target solutions $z$ which do not exhibit any of these features are called \mydef{simple} and are characterized by the invertibility of the Jacobian $\textrm{Jac}(F)|_z$ of $F$ at $z$. Handling the non-simple target solutions is done computationally by~\mydef{endgame algorithms} \cite{NAGbook}.

If the number of target solutions equals the number of start solutions which must be tracked, then the homotopy is called \mydef{optimal}. Non-optimality is traditionally accounted for by the  paths which exhibit one of the three non-simple endgame behaviors. \autoref{fig:homotopypic} illustrates a non-optimal homotopy for which is the target system has three fewer solutions than the start system: one solution path diverges, one solution path is redundant, and one solution path approaches a non-isolated solution. If multiplicity is counted, the target system only has two fewer solutions that the start system.

\bigskip
\section{Iterators}
\label{sec:iterators}

Recall from the introduction that an \mydef{iterator} is an object which represents a set  
\begin{equation*}
A = \{a_1,\ldots,a_d\}
\end{equation*} 
via the ability to iterate through its elements \cite{Watt}. In our case, the iterator for the target solutions~$T$ (the zeros of $F$) is defined by an iterator \code{I} for the start solutions $S$ (the zeros of $G$) \emph{and} a homotopy $H$ from $F$ to~$G$. A different iterator for $S$ or homotopy may result in a different iterator for~$T$. 
\begin{definition}\label{def_hom_iter}
Given a homotopy $H$ which connects the start solutions $S$ to the target solutions $T$ and an iterator \code{I} for the start solutions, the pair $(H, \text{\code{I}})$, represents the  iterator $\text{\code{J}}  = f_H(\text{\code{I}})$ for $T$. We call $\text{\code{J}}$ a \mydef{homotopy iterator} for $T$. 
\end{definition}
In most cases, we will drop the formalism of this definition when the homotopy is clear from the context. In some cases, however, it is important to underline the role of the homotopy $H$ and we may describe the homotopy iterator by the pair $(H, \text{\code{I}})$. Alternatively, we may write $f_H(\text{\code{I}})$ to highlight that homotopy iterators are push-forwards of functions, as described in the next section.

An iterator implicitly represents the \emph{set} $A$ by an \emph{ordered set}, or list $a_1, \ldots, a_d$.
In particular, we assign to each element $a_i$ a \mydef{state} $i$.
In \code{Julia}, one implements an iterator by defining a \code{struct iter} and a method \code{iterate(I::iter, state::Int)} which takes in an iterator \code{I} and a state \code{i} and returns the next element and state \code{(a_(i+1), i + 1)}.
The iterator represents $A$ without holding $A$ in memory.

For example, an iterator for the $n$-th roots of unity might be implemented  as 
\begin{lstlisting}[language=code]
struct RootsOfUnity{T <: Integer}
    n::T
end
Base.iterate(I::RootsOfUnity) = (1.0 + 0.0*im, 0) 
function Base.iterate(I::RootsOfUnity, state::Int)
    next = state + 1
    next >= I.n ? nothing : (exp(2pi*im*(state+1)/I.n), next)
end
\end{lstlisting}
Here, the $(i+1)$-st root of unity is evaluated using the state $i$ instead of the $i$-th root of unity.
When called on the last state, \code{iterate} returns \code{nothing} indicating that there are no more elements.

\subsection{Manipulating iterators}\label{iter_operations}

Given an iterator \code{I} for a set $A$, many operations on $A$ may be performed in a memory efficient way. 
One example is that of \mydef{accumulation}. Given an initial value $c$ in some set $C$ and a function $f:A \times C \to C$, the accumulation of $f$ over $A$ starting at $c$ is 
$$\big(f(a_d, \, \underline{\quad}\,  ) \circ f(a_{d-1}, \, \underline{\quad}\,  ) \circ \cdots \circ f(a_2, \, \underline{\quad}\,  )\circ f(a_1, \, \underline{\quad}\,  )\big)(c).$$
Programmatically, one may write
\begin{lstlisting}[language=code]
function accumulate(f, I, c)
    for item in I
        c = f(c, item)
    end
    return c
end
\end{lstlisting}
This can be evaluated using an iterator for \code{I} by only holding $c$ and one element of $A$ in memory at once. The \code{Iterators} module in \code{Julia} \cite{julia} provides an accumulate function. The above function would be \code{accumulate(f, I; init = c)}, with the difference that the \code{accumulate} function in \code{Julia}~\cite{julia} returns the vector of all intermediate values, not just the last element like our implementation. 
Other familiar functions can be interpreted as accumulation:
\begin{itemize}
\item The maximum $\textrm{max}_{a \in A}\, f(a)$ as:
\begin{lstlisting}[language=code]
max(f, I) = accumulate((c, a) -> max(f(a), c), I, -Inf)
\end{lstlisting}
\item The sum $\sum_{a \in A}{f}(a)$ as:
\begin{lstlisting}[language=code]
sum(f, I) = accumulate((c, a) -> f(a) + c, I, 0)
\end{lstlisting}
\item The number of elements $|A|$ as:
\begin{lstlisting}[language=code]
count(I) = sum(a -> 1, I) = accumulate((c, a) -> 1 + c, I, 0) 
\end{lstlisting}
\item The number of elements  $|\{a \in A \mid {f}(a)=\textrm{true}\}|$ satisfying some condition as: 
\begin{lstlisting}[language=code]
conditional_count(f, I) = sum(a -> f(a) ? 1 : 0, I)
\end{lstlisting}
\end{itemize}

When $A=S \subseteq \mathbb{C}^n$ is a solution set to a polynomial system  \code{sum(a -> a, I)} represents the \textit{trace} of $S$ \cite{TraceTest,SparseTraceTest}, \code{count(I)} represents the number of complex solutions, and \code{conditional_count(is_real, I)} represents the number of real solutions.

New iterators can be constructed from old ones in several ways. Suppose again that \code{I} is an iterator for $A = \{a_1, \ldots, a_d\}$ that $f:A \to C$ is some function. The list 
$$\mydef{f(A)} = (f(a_1),\ldots,f(a_d))$$
can easily be represented by the push-forward iterator \code{f(I)} which holds an element $a$ of $A$ as its state and returns $f(a)$ as its value. This operation justifies the notation of $f_H(\text{\code{I}})$ in \autoref{def_hom_iter}. The \code{iterate} function for \code{f(I)} is essentially the same as that of \code{I}. In \code{Julia} \cite{julia}, this~is 
\begin{lstlisting}[language=code]
Iterators.map(f, I)
\end{lstlisting}
\code{Iterators.map} is the lazy version of the function \code{map}. 

If $f: A \to \{0,1\}$ is a boolean-valued function, then the set $\mydef{A|f} = \{a \in A \mid f(a) = 1\}$ filters out those entries for which $f$ evaluates to false.  
One may implement an iterator \code{J} for $A|f$ in two ways. The first requires an iterator \code{I} for $A$. Then, one defines \code{iterate(J, i)} 
to be \code{iterate(I,i)} if $f(a_i)=1$ and \code{iterate(I, i+1)} otherwise. 
This approach requires no computation of $f$ until a user wants to access elements of $A|f$.  We can implement this in \code{Julia} \cite{julia} as 
\begin{lstlisting}[language=code]
J = Iterators.filter(f, I)
\end{lstlisting}
Alternatively, a user may precompute a bit-vector $b = f(A)$ so that evaluation of $f$ amounts to merely accessing the $i$-th element of the vector $b$. This bit-vector is called a \mydef{bitmask} for $f$ applied to \code{I}. As discussed in the introduction, we implemented \code{bitmask} functionality in our setting via
\begin{lstlisting}[language=code]
J = bitmask_filter(f, I)
\end{lstlisting}

The concatenation of two iterators \code{I} and \code{J} is obtained in \code{Julia} \cite{julia} via 
\begin{lstlisting}[language=code]
Iterators.flatten((I, J))
\end{lstlisting}
To form an iterator for the product of \code{I} and \code{J} one uses 
\begin{lstlisting}[language=code]
Iterators.product(I, J)
\end{lstlisting}
To run iterators \code{I} and \code{J} at the same time, until either of them is exhausted, we iterate through 
\begin{lstlisting}[language=code]
Iterators.zip(I, J)
\end{lstlisting}
If  \code{I} is an iterator for $A=\{a_1,\ldots, a_d\}$ and \code{J} is an iterator for $B=\{b_1,\ldots, b_e\}$, then the zip iterator is an iterator for the maximal diagonal set $\{(a_1, b_1),\ldots,(a_m, b_m)\}$, where $m=\min\{d,e\}.$

\subsection{Composition of homotopy iterators}\label{composition}
Most of the operations in the previous section can be applied to any iterator. A particularly interesting case in the context of homotopy iterators is \mydef{composition} of push-forward operations. This works as follows. Suppose that $H$ and $H'$ are homotopies so that $H(\x; 0)=H'(\x; 1)$. Then the target solutions of $H$ may be used as start solutions to $H'$. On the level of iterators, if $\text{\code{I}}$ represents the start solutions to $H$, then 
\[
f_{H'}(f_{H}(\text{\code{I}})) = (f_{H|H'})(\text{\code{I}})
\]
where $H|H'$ is the \mydef{concatenation of homotopies}
\[
H|H'(\x;t) = \begin{cases}
H(\x;2t-1) & t \in [1/2,1] \\
H'(\x; 2t) & t \in [0,1/2]
\end{cases}
\]

As a consequence of our implementation, which abstracts solution sets from vectors (of solutions) to iterators (for solutions), concatenation of homotopies works elegantly out-of-the-box. Here is an example involving parameter homotopies. We use a vector of start solutions as the start solution iterator \code{I}. Note that a vector can be iterated over, and is thus an iterator itself.

\begin{lstlisting}[language = code]
using HomotopyContinuation  
@var x y p
f = [y - x^2 + p; y - x^3 - p]
F = System(f; variables = [x; y], parameters = [p])
I = [[1, 1], [-1, 1]]
J = solve(F, I;
          start_parameters = [0], target_parameters = [-1],
          iterator_only = true)
\end{lstlisting}
Now, \code{J} is an iterator for the first homotopy that tracks \code{F} from $p=0$ to $p=-1$. We pass the iterator  \code{J} to \code{solve} obtaining an iterator for the parameter homotopy from $-1$ to~$-2$:
\begin{lstlisting}[language = code]
K = solve(F, J;
          start_parameters = [-1], target_parameters = [-2],
          iterator_only = true)
\end{lstlisting}
Recall that this entire code block \emph{never} tracks a single path. Nevertheless, we have obtained a composed iterator for the concatenated homotopy. 
Composition of iterators can itself be iterated, so that we can compose any number of subsequent homotopies. This can be used to encode monodromy loops via iterators. 

\bigskip
\section{Start solution iterators}
\label{sec:start_system_iterators}

The definition of a homotopy iterator (\autoref{def_hom_iter}) requires an iterator for the solutions of the start system in a homotopy. There are several popular choices for a start system $G(\x)$ in a homotopy \eqref{eq:homotopy}. We list some here and discuss how to construct iterators for their solution sets.

\subsection{Total degree start solution iterators}\label{sec:td_start_system}
The basic choice for a start system $G(\x)$ in a straight-line homotopy \eqref{eq:homotopy} is the \mydef{total degree start system}:
\begin{equation}
\label{eq:total_degree}
G(\x) = \begin{bmatrix} g_1(\x) \\ g_2(\x) \\ \vdots \\ g_n(\x) \end{bmatrix}= \begin{bmatrix} x_1^{d_1}-1 \\ x_2^{d_2}-1 \\ \vdots \\ x_n^{d_n}-1 \end{bmatrix} = \textbf{0}, \quad \quad d_i = \textrm{deg}(f_i).
\end{equation}
The number of solutions to $G(\x)=\textbf{0}$ is  $d = \prod_{i=1}^{n} d_i$, which is called the \mydef{B\'ezout bound} of $F$ in reference to B\'ezout's theorem. When it comes to start systems, the total degree start system has the largest solution count, and thus the highest likelihood of being non-optimal. On the flip-side, the solutions are extremely easy to write down: they are $n$-tuples of $d_i$-th roots of unity and thus can be described by a product of iterators. At the beginning of \autoref{sec:iterators}, we explain how to write an iterator for the $d_i$-th roots of unity, and in \autoref{iter_operations}, we explain the product of iterators. The total degree start solution iterator implicitly puts the tuples of roots of unity in bijection with the numbers $1,2,\ldots,d$. The inverse of this bijection,  called \code{bezout_index}, is used in~\autoref{sec:compression}.

In \code{HomotopyContinuation.jl} \cite{HC} one can set up a homotopy iterator for solving a system $F$, that uses the straight-line homotopy together with the total degree start system, as follows.
\begin{lstlisting}[language = code]
I = solve(F; start_system = :total_degree, iterator_only = true)
\end{lstlisting}

\subsection{Polyhedral start solution iterators}\label{sec:polyhedral_start_system}
Another way to construct a start system for $F(\x)$ is to consider its monomial support
\[
\mathcal A_{\bullet}=(\mathcal A_1,\ldots,\mathcal A_n) \text{ where } \mathcal A_i = \{\alpha \in \mathbb{Z}^n \mid [\textbf{x}^{\alpha}]f_i \neq 0\}.
\]
The notation $[\x^{\alpha}]f_i$ denotes the coefficient of the monomial $x_1^{\alpha_1}\cdots x_n^{\alpha_n}$ in the polynomial $f_i$. Each $\mathcal A_i$ is a finite set.
One can then consider $F(\x)$ as a \mydef{sparse polynomial system}, that is, a single member of the family of all polynomial systems with the monomial support $\mathcal A_{\bullet}$. Just as B\'ezout's theorem provides an upper bound for the number of solutions to $F(\x)$ based on the degrees $(d_1,\ldots,d_n)$ of $f_1,\ldots,f_n$, the \mydef{Bernstein-Kouchnirenko-Khovanskii} theorem provides an upper bound on the number of solutions to $F(\x)=0$ in $(\mathbb{C}^\times)^n$ based on the refined information of the supports $\mathcal A_{\bullet}$. This number, which we denote by $d_{\mathcal A_{\bullet}}$ is called the \mydef{BKK bound} of $F(\x)$. 

The BKK bound $d_{\mathcal A_{\bullet}}$ can be evaluated via a polyhedral computation on the \mydef{Newton polytopes} $$P_i = \textrm{convexHull}(\mathcal A_i).$$
Specifically, 
$$d_{\mathcal A_{\bullet}} = \mathrm{MixedVolume}(P_1,\ldots,P_n).$$
One important characterization of mixed volume is as follows. The mixed volume $d_{\mathcal A_{\bullet}}$ is  the sum of the volumes of the \textit{mixed cells} in a \textit{mixed subdivision} of the \mydef{Minkowski sum}
\[
\mathcal A_1+\cdots + \mathcal A_n = \{{\bf a}_1+\cdots+{\bf a}_n \mid {\bf a}_i \in \mathcal A_i\}.
\] 
It is via this combinatorial interpretation of $d_{\mathcal A_{\bullet}}$ that the \mydef{polyhedral homotopy} of Huber and Sturmfels \cite{HS95} computes the solution set to a polynomial system $F(\x)=\textbf{0}$ supported on $\mathcal A_{\bullet}$. The polyhedral homotopy takes the following steps:
\begin{enumerate}
\item Find an appropriate lift $\omega:\mathcal A_{\bullet} \to \mathbb{Z}$ of the supports to induce a mixed subdivision of $\mathcal A_{\bullet}$ involving mixed cells $C_1,\ldots,C_k$, each associated to a vector $\nu_1,\ldots,\nu_k$.\smallskip
\item For each mixed cell $C_i$, construct an easy-to-solve binomial system $B_i({\bf x}) = \bf 0$ in new coordinates which has ${\rm vol}(C_i)$ many solutions, the set of which is denoted $S_i$. \smallskip
\item Construct a homotopy $H_{i}$ which connects the solutions $S_i$ of the binomial system $B_i({\bf x}) = \bf 0$ to $\textrm{vol}(C_i)$ many solutions of $F(\x)=\textbf{0}$.\smallskip
\item Follow the $\textrm{vol}(C_i)$ many solutions of the start system $B_i({\bf x})=\textbf{0}$ along the homotopy $H_i$ to find the same number of solutions to $F(\x)=\textbf{0}$. \smallskip
\end{enumerate}
Therefore the solution set to $F(\x) = \bf 0$ is included in $\bigcup_{i = 1}^k f_{H_i}(S_i)$. When $F$ is not generic, then the number of zeros of $F$ may be strictly smaller than $d_{\mathcal A_{\bullet}}$ and endgames must be used to sort out those paths that do not converge to solutions of $F$. In practice, one generates a random generic system $G$ with support $\mathcal A_{\bullet}$, find zeros of $G$ as described above, and tracks these to zeros of $F$. 

The start system iterator \code{I} can thus be constructed in three steps: First, create an iterator \code{I_MC} for the mixed cells. Then, define a function $f$ that maps a cell $C_i$ to the homotopy iterator $(H_i, \text{\code{I_i}})$, where \code{I_i} is an iterator for $T_i$. Finally, flatten the iterator $f(\text{\code{I_MC}})$. 

In \code{HomotopyContinuation.jl} \cite{HC} one can set up a homotopy iterator for solving a system $F$ that uses a polyhedral start system iterator, as follows.
\begin{lstlisting}[language = code]
I = solve(F; start_system = :polyehdral, iterator_only = true)
\end{lstlisting}
In fact, one can leave out the flag that sets the polyhedral homotopy, since it is the default choice. Moreover, \code{HomotopyContinuation.jl} \cite{HC} will define the two step homotopy that first tracks to a generic system $G$ and then to $F$. On the level of iterators, this is a composition as discussed in~\autoref{composition}. In \autoref{secsec:chemistry} we construct an explicit polyehdral start solution iterator for an example system.

\subsection{Parameter homotopy start solution iterators}
As already mentioned in \autoref{sec:polynomial_homotopy_continuation}, many polynomial systems naturally belong to a family indexed by some parameter space. That is, $F(\x)$ is one instance $F(\x; \p^*)$ of a square \mydef{parametrized polynomial system} 
\begin{equation}
\label{eq:parametrized_system}
F(\x;\p) = \begin{bmatrix} f_1(\x;\p) \\ f_2(\x;\p) \\ \vdots \\ f_n(\x;\p) \end{bmatrix} = \textbf{0}, \quad \quad f_1,\ldots,f_n \in \mathbb{C}[p_1,\ldots,p_k][x_1,\ldots,x_n].
\end{equation}
in $k$ parameters $\p=(p_1,\ldots,p_k)$, $n$ variables $\x=(x_1,\ldots,x_n)$ and $n$ equations $f_1,\ldots,f_n$. It is useful to encode this family via its \mydef{incidence variety}
\[\mathcal I = \{(\x,\p) \mid f_1(\textbf{x};\p) = \cdots = f_n(\textbf{x};\p) = 0 \} \subset \mathbb{C}^n \times \mathbb{C}^{k}
\]
which encodes all parameter-solution pairs to \eqref{eq:parametrized_system}. 
One can visualize the natural projection map ${\pi}: \mathcal I \to \mathbb{C}^{k}$ onto the \mydef{parameter space} $\mathbb{C}^k$  as shown in \autoref{fig:branched_cover},  and interpret the fibres $\pi^{-1}(\p^*)$ as the solutions to \eqref{eq:parametrized_system} specialized at $\p = \p^*$. 
Given any $\p^* \in \mathbb{C}^k$, write ${d_{\p^*}}$ for the number of isolated simple solutions to $F(\x; \p^* )= \mathbf 0$.

The \mydef{parameter continuation theorem} (see \cite{BOROVIK2025102373, MS1989}) states that there exists a proper subvariety ${\Delta} \subset \mathbb{C}^k$ of parameters called the \mydef{exceptional set} of $\pi$ containing all $\p^*$ for which $d_{\p^*}$ is not equal to ${\textrm{deg}(\pi)} = \sup_{p \in \mathbb{C}^k}(d_{p})$. Moreover, it states that  $0<\textrm{deg}(\pi)<\infty$. The complement of $\Delta$ is the open subset ${\mathcal U}$ of \mydef{regular values} of $\pi$, which are also sometimes called \mydef{generic parameters}.
\begin{figure}[!htpb]
\includegraphics[scale=0.575]{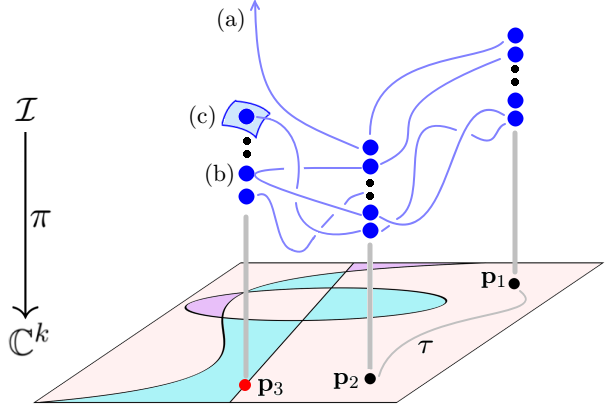}
\caption{A visualization of the branched cover $\pi: \mathcal I \to \mathbb{C}^{k}$ from the incidence variety to the parameter space.
}
\label{fig:branched_cover}
\end{figure}

Another interpretation of the exceptional set $\Delta$ is that it is the Zariski closure of the set of parameter values for which a homotopy from a generic parameter exhibits poor endgame behavior as outlined in \autoref{sec:polynomial_homotopy_continuation}. This point-of-view is illustrated in \autoref{fig:branched_cover} where $\mathbf p_3$  is a parameter value which has each endgame feature.

An iterator for the solutions to $F(\x;\p^*)=\textbf{0}$ for \emph{any} generic $\p^*$ may be used as a start solution iterator for a parameter homotopy. Continuing the themes of the article, such a fibre need not be explicitly held in memory, but can be represented by a homotopy iterator from one of the previous natural start systems. \autoref{sec:compression} explains how to obtain an optimal homotopy for such a fibre when the start system is the total degree system.

\subsection{Combinatorial start solution iterators}
Another way to obtain a start solution set for a parameter homotopy is through its combinatorics.  Many classical enumerative problems have combinatorial machinery for determining their solution counts. Examples include  sparse polynomial systems and Schubert calculus. In a subset of these instances, the combinatorial count extends to an explicit natural bijection $\varphi: \mathfrak C \to S$ between some set $\mathfrak C$ of combinatorial gadgets and a solution set $S \subseteq \mathbb{C}^n$ to $F(\x; \p^*) = \textbf{0}$ for a generic parameter $\p^* \in \mathbb{C}^k$. In this case, an iterator for $\mathfrak C$ can be pushed forward via $\varphi$ to obtain an iterator for $\pi^{-1}(\p^*)$, that is, a start solution iterator. We hope that the power of combinatorial start systems inspires further searches for such bijections. 

We illustrate this idea with a single example coming from an interpolation problem considered in \cite{Necklaces}. The problem asks for the polynomially parametrized bi-degree $(d_1,d_2)$ of a curve in $\mathbb{C}^2$ which meets a generic germ $f = \sum_{i=1}^\infty c_i x^i$ at $x=0$ to as high an order as possible, namely $d_1+d_2-1$. The interpolating curve is represented as 
\[
t \mapsto (x(t),\ y(t)), \quad \text{where } x(t) = a_1t+a_2t^2+\cdots+a_{d_1}t^{d_1},\ y(t)= b_1t+b_2t^2+\cdots+b_{d_2}t^{d_2},
\]
and the equations are 
\[
h_i(\textbf{a},\textbf{b}) = \text{coefficient of $t^i$ in } (y(t)-f(x(t))) = 0 \quad \text{ for } t=1,\ldots,d_1+d_2-1
\]
resulting in a  polynomial system in $d_1+d_2$ parameters $c_1,\ldots,c_{d_1+d_2}$,  $d_1+d_2$ variables $a_1,\ldots,a_{d_1}$, $b_1,\ldots,b_{d_2}$, and $d_1+d_2-1$ equations $h_1,\ldots,h_{d_1+d_2-1}$. There is a weighted homogeneity corresponding to a reparametrization $t \mapsto \alpha t$ which can be made finite by the condition $a_{d_1}\cdot a_{d_2}=1$. Alternatively, one may select only the smooth interpolants by setting $a_1=1$, as is done in the following code
\begin{lstlisting}{language=code}
function polynomial_interpolants(d1, d2)
  d = d1 + d2
  @var c[1:d]; @var a[1:d1]; @var b[1:d2]; @var t
  x = sum([a[i]*t^i for i in 1:d1])
  y = sum([b[i]*t^i for i in 1:d2])
  f = y - sum([c[i]*x^i for i in 1:d])
  C = coefficients(f, t) |> reverse
  System([C[1:d-1]; a[1] - 1], variables = vcat(a, b), parameters = c)
end
\end{lstlisting}

One main result of \cite{Necklaces} is that the degree of this problem is equal to the cardinality $\mathcal N_{d_1,d_2}$ of the set $\mathfrak C$ of aperiodic binary necklaces on $d_1$ white beads and $d_2$ black beads. The stronger result of that article is a bijection $\varphi$ between $\mathfrak C$ and the fibre $\pi^{-1}(\textbf{c})$  over the germ of the function 
\[
y=\frac{1}{x+1}-1 = -x+x^2-x^3+x^4+\cdots \quad \quad \textbf{c} = (-1,1,-1,\ldots)
\]
The bijection is as follows. Superimpose any necklace onto the $d_1+d_2$ roots of $-1$ partitioning them into sets $\{\alpha_1,\ldots,\alpha_{d_1}\}$ and $\{\beta_1,\ldots,\beta_{d_2}\}$ of white and black roots of $-1$. Then construct the polynomials 
\[
x(t) = -1+\prod_{i=1}^{d_1}(\alpha_it+1) \quad \text{ and } \quad y(t) = -1+\prod_{i=1}^{d_2}(\beta_it+1).
\]
Such curves are solutions since 
\[
\frac{1}{(x(t)+1)(y(t)+1)} \equiv 1 \textrm{ mod } t^{d_1+d_2} \iff \prod_{i=1}^{d_1}(\alpha_it+1)\prod_{i=1}^{d_2}(\beta_it+1) = 1+t^{d_1+d_2}
\]
This construction does not take into account rotation of the roots of unity, which corresponds to the finitely many reparametrizations $t \mapsto e^{2\pi \sqrt{-1}/(d_1+d_2)}t$, and so one only need to take a single representative of each such equivalence class. The following code functions as $\varphi$:

\begin{lstlisting}{language=code}
function solution_from_necklace(W, B)
	@var t; d = length(W) + length(B)
	x = prod([exp(2*pi*im*k/d)*t + 1 for k in W]) - 1
	y = prod([exp(2*pi*im*k/d)*t + 1 for k in B]) - 1
	avec = coefficients(x, t) |> reverse
	bvec = coefficients(y, t) |> reverse
	reparam = 1/avec[1]
	avec_smooth = [avec[i]*reparam^i for i in 1:length(W)]
	bvec_smooth = [bvec[i]*reparam^i for i in 1:length(B)]
	[avec_smooth; bvec_smooth]
end
\end{lstlisting}

\begin{figure}
\includegraphics[scale=0.3]{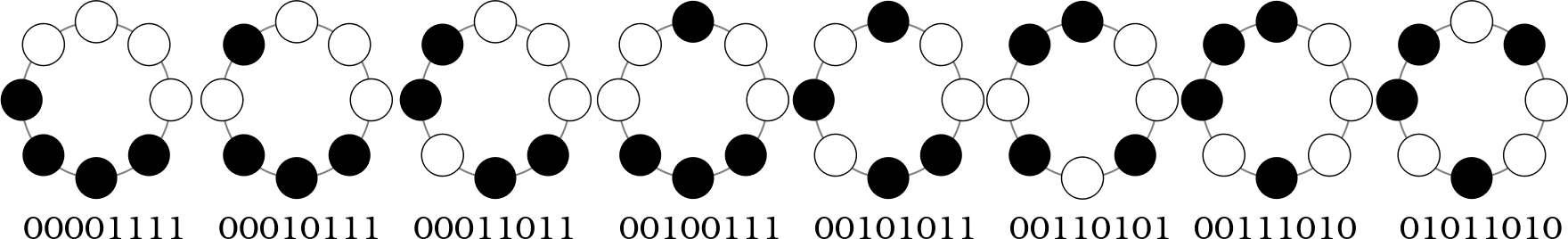}
\caption{A depiction of the 8 aperiodic necklaces on four white and four black beads used as a combinatorial indexing set for a set of start solutions.}
\label{fig:necklaces}
\end{figure}
Now, for any iterator \code{C} for the set $\mathfrak C$, one may push $\mathfrak C$ forward via the above function to produce a combinatorial start solution iterator. 

For example, for $d_1=d_2=4$, there are eight aperiodic $(4,4)$-necklaces (see \autoref{fig:necklaces}). The following code solves the system for a generic germ in the homotopy iterator version of the usual two-step parameter homotopy. 
\begin{lstlisting}{language=code}
d1, d2 = 4, 4; d = d1 + d2
F = polynomial_interpolants(d1, d2)
gen_parameters = randn(ComplexF64, d); target_parameters = randn(Float64, d)
M = [0 0 0 0 1 1 1 1;  0 0 0 1 0 1 1 1;  0 0 0 1 1 0 1 1;  0 0 1 0 0 1 1 1;
     0 0 1 0 1 0 1 1;  0 0 1 1 0 1 0 1;  0 0 1 1 1 0 1 0;  0 1 0 1 1 0 1 0]
C = [[findall(x -> x==i, r) for i in 0:1] for r in eachrow(M)]
I = map(N -> solution_from_necklace(N...), C)
J_intermediate = solve(F, I; iterator_only = true,
                          start_parameters=[(-1)^i for i in 1:d], 
                          target_parameters = gen_parameters)
J_final = solve(F, J_intermediate; iterator_only = true,
                          start_parameters = gen_parameters,
                          target_parameters = target_parameters)
\end{lstlisting}
 
In \autoref{secsec:chemistry}, we provide another illustrative example of how combinatorial iterators can be pushed forward to solution iterators. In that setting, we have a combinatorial interpretation for the mixed cells which we may use as an iterator for the start solutions of a polyhedral homotopy.

 \bigskip
 \section{Implementation, applications, and examples}\label{sec:implementation}
 \label{sec:implementation_applications_examples}
 \subsection{Implementation in HomotopyContinuation.jl and Functionality}
Given the modularized structure of \code{HomotopyContinuation.jl} (\code{HC.jl}) \cite{HC}, the implementation of homotopy iterators was fairly straightforward. We introduce a new data type called a \code{ResultIterator} that encodes homotopy iterators. The name is motivated by the data structure \code{Result} that \code{HC.jl} uses to encode the output of its \code{solve} function. 

The object \code{ResultIterator} has three fields: a \code{Solver}, a \code{StartSolutionsIterator}, and optionally a \code{bitmask}. The first two are existing datatypes in \code{HC.jl} and the last field is implemented as a \code{BitVector}, a datatype which is native in \code{Julia} \cite{julia}. A  \code{Solver} in \code{HC.jl} holds a homotopy together with other various information (such as endgames and tracker settings). A  \code{StartSolutionsIterator} is a datatype that encodes start solution iterators. Therefore, a \code{ResultIterator} implements our idea of a homotopy iterator (\autoref{def_hom_iter}). Note that this implementation is not entirely memoryless, since computing and storing both a \code{Solver} and a \code{StartSolutionsIterator} consumes memory. 

We implemented \code{ResultIterator} as a subtype of \code{AbstractResult}, which is the suptype of all objects that the \code{HC.jl} main function \code{solve} returns as its output. This way \code{ResultIterator} integrates automatically in the \code{HC.jl} ecosystem -- thanks to multiple dispatch in \code{Julia}.

\subsection{Illustrating our implementation in an example} 
 The problem of $3264$ conics is a famous enumerative problem which asks for the $3264$ conics tangent to five general conics in the plane. A  set of equations for this problem in the affine chart where no conic passes through the origin is easy to define in \code{HomotopyContinuation.jl}:
\begin{lstlisting}[language=code]
using HomotopyContinuation, LinearAlgebra   
@var x, y, a[1:5, 1:6], b[1:5],  u[1:5], v[1:5]
FiveConics = [a[i,1]*x^2 + a[i,2]*x*y + a[i,3]*y^2 + a[i,4]*x + a[i,5]*y + 1 
                for i in 1:5]
SolutionConic = b[1]*x^2 + b[2]*x*y + b[3]*y^2 + b[4]*x + b[5]*y + 1 
function steiner_condition(i)
    vars = [u[i], v[i]]
    S = evaluate(SolutionConic, [x, y] => vars)
    C = evaluate(FiveConics[i], [x, y] => vars)
    D = det([differentiate(S, vars) differentiate(C, vars)])
    [S, C, D]
end
Eqs = map(steiner_condition, 1:5)
F = System(reduce(vcat, Eqs), variables = vcat(b, u, v), parameters = vec(a))
\end{lstlisting}
On a standard laptop monodromy solve takes approximately $60$ seconds to solve this enumerative problem without \emph{a priori} knowledge of the number of solutions. A polyhedral homotopy takes approximately $140$ seconds, and a total degree homotopy, approximately $450$ seconds. 

The construction of a homotopy iterator induced by a polyhedral iterator is as costly as the construction of the tracker and the start solution iterator:
\begin{lstlisting}[language=code]
tp = randn(Float64,30)
@time I = solve(F, target_parameters = tp; iterator_only = true)
0.029489 seconds (237.08 k allocations: 9.072 MiB)
  ResultIterator
  ==============
 * start solutions: PolyhedralStartSolutionsIterator
 * homotopy: Polyhedral
\end{lstlisting}

Target solutions are represented as another iterator, filtered by those path results which are finite and non-singular:
\begin{lstlisting}[language=code]
julia> @time solutions(I)
Warning: Since result is a ResultIterator, counting multiple results
0.000046 seconds (8 allocations: 880 bytes)
\end{lstlisting}
The warning here indicates that this iterator will not check whether there are double solutions coming from non-simple zeros. By constrast, the usual \code{solve} function in \code{HC.jl} will check the results on double solutions. 
Solutions can then be collected, which triggers the actual computation of continuation upon the start solutions:
\begin{lstlisting}[language=code]
julia> @time collect(solutions(I))
Warning: Since result is a ResultIterator, counting multiple results
142.727156 seconds (599.95 k allocations: 65.479 MiB, 0.05% gc time)
3264-element Vector{Vector{ComplexF64}}...
\end{lstlisting}

\begin{figure}[!htpb]
\includegraphics[scale=0.15]{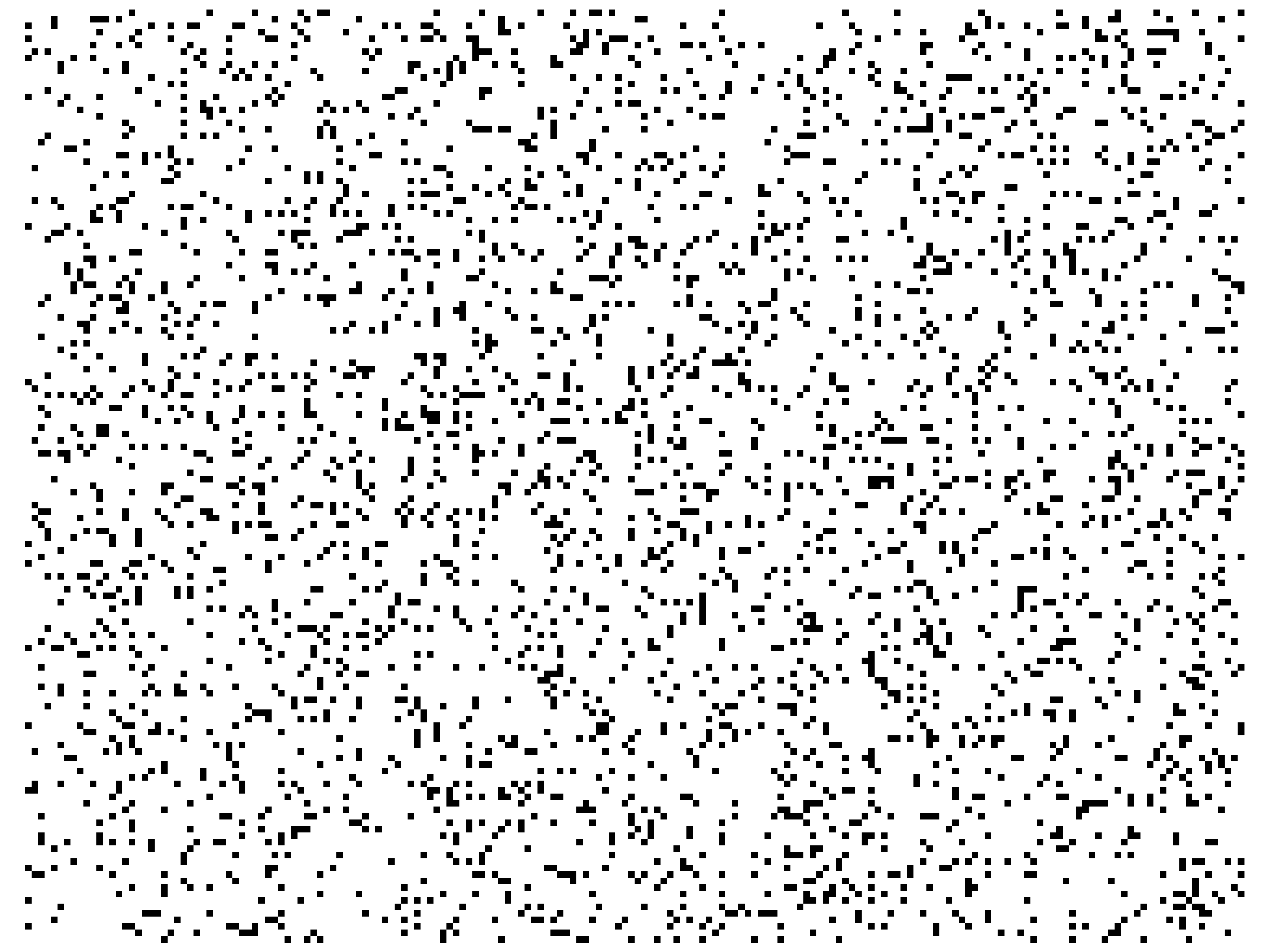}
\caption{A $144 \times 188$ matrix illustrating the length $27072$ bitmask indicating which of the polyhedral start solutions are tracked to solutions to Steiner's conic problem. The black pixels are start solutions which map to finite nonsingular target solutions. The white pixels are not. }
\label{fig:bitmask}
\end{figure}

Alternatively, instead of collecting all solutions, we could produce a bitmask of those start solutions which are finite and non-singular. 
\begin{lstlisting}[language=code]
julia> @time B = bitmask_filter(s -> is_success(s) && is_nonsingular(s), I)
142.011248 seconds (444.42 k allocations: 52.537 MiB)
ResultIterator
==============
*  start solutions: PolyhedralStartSolutionsIterator
*  homotopy: Polyhedral
*  filtering bitmask
\end{lstlisting}
The bitmask in this case is a $27072$ length bit-vector with $3264$ one's corresponding to the $3264$ start solutions which map to finite non-singular target solutions. Such a bit-vector can be easily stored, for example,  \autoref{fig:bitmask} depicts the vector as black dots in a 2D image.
Note also that the bitmask only requires $27072 \textrm{ bits} \approx 26.43\textrm{kb}$
whereas the $3264$ solutions in $\mathbb{C}^{15} \cong \mathbb{R}^{30}$ requires $30 \cdot 64 \cdot 3264 = 6266880 \textrm{ bits} = 765\textrm{kb}$. This bitmask could be transmitted to a recipient, along with the datum of a polyhedral homotopy, who would then know \emph{a priori} which solutions are necessary to track.

 \subsection{Total degree compression}\label{sec:compression}

Given a solution set, homotopy iterators can be used to compress the solutions for optimal unpacking at a later time. Suppose, for instance, that we have represented the solution set $T$ to a system $F(\x)=\textbf{0}$ by an iterator \code{J}. 
To compress \code{J} we could track it to the total degree start system $G$ as $t$ goes from $0$ to $1$ in the straight-line homotopy $H = (1-t)F +  t\gamma G$, where $\gamma \in\mathbb C^\times$ is random, and keep track of which total degree start solutions are found. 
For this we need a function \code{bezout_index}, which is the inverse function of the implicit ordering given by a start solution iterator for the total degree start system. 
\begin{lstlisting}[language=code]
function indices_of_entries(z, d)
    args = map(angle, z) ./ (2pi) 
    map(zip(args, d)) do (ai, di)
      bi = round(ai * di) |> Int
      mod(bi, 0:di-1)
    end
end
function bezout_index(z, d)
  ind = indices_of_entries(z, d)
  l = length(d)
  BI = sum(prod(d[j] for j in 1:i-1) * ind[i] for i in 2:l) + ind[1] + 1
  Int(BI)
end
\end{lstlisting} 
This function is then used in a compress function that realizes the above idea for compression:
\begin{lstlisting}[language=code]
function compress(F, J; gamma = cis(rand() * 2pi))
  d = degrees(F); v = variables(F);
  G = System(gamma .* [vi^di - 1 for (vi, di) in zip(v, d)], variables = v)
  I = solve(F, G, J; iterator_only = true)
  S_F = Iterators.map(solution, I) # solutions of G from F
  S_G = total_degree_start_solutions(d) # all solutions of G
  ind = Iterators.map(s -> bezout_index(s, d), S_F) # Bezout index
  B = falses(prod(d)); for i in ind; B[i] = 1; end
  solve(G, F, S_G; iterator_only = true, bitmask = B)
end
\end{lstlisting}
In the above function, the iterator \code{BI} represents the total degree  start solutions which are tracked to $T$ under $H$.  Moreover, the solutions $T$ are represented as an iterator in the output \code{J}. Crucially, it is the bitmask \code{B} which encodes the solutions to \code{F}.

Now, we can run compression on a solved system $F(\x)=\textbf{0}$ with solutions $T$ as follows:
\begin{lstlisting}[language=code]
T = solve(F)
C = compress(F, T)
\end{lstlisting} 
We refer to the iterator \code{C} as the \mydef{total degree compression} of the solutions $T$. 
Communicating \code{C} via email requires only the transmission of $F$ and $B$. The total degree start system $G$ is implicit, as is the straight-line homotopy $H$. We remark that it is possible that the homotopy is not generic. We can unpack $T$ by collecting this iterator:
\begin{lstlisting}[language=code]
solutions_of_F = collect(C)
\end{lstlisting} 

A fantastic feature of compression is that the solution set \code{T} need not be held in memory all at once either. For example, to obtain the same total degree compression one could run:
\begin{lstlisting}[language=code]
I = solve(F; iterator_only = true, start_system = :polyhedral)
J = Iterators.filter(s -> is_success(s) && is_nonsingular(s), I)
C = compress(F, J)
\end{lstlisting} 

As an extreme example, consider the problem of computing $d=10^{10}$ solutions in $\mathbb{C}^{100}$ to a system of mixed volume $d$ and B\'ezout bound $10^{12}$. Storing this many solutions is infeasible and requires about $16$ Terabytes. Tracking via the polyhedral start system is optimal, but perhaps the initialization of a polyhedral start system requires significant computation by a user, say Alice, who has the time and resources to do so. Alice's peer Bob, on the other hand, does not have the time and resources to repeat it. Without homotopy iterators, Alice would not be able to represent these $d$ solutions, or effectively communicate about them to Bob. With a homotopy iterator, Alice could run the above code and produce a total degree compression iterator for Bob, despite her own memory restrictions. Doing so makes communication between Alice and Bob about the $124^{\textrm{th}}$ solution, for example, possible.

Finally, we remark on another application of tracking solutions in the less-frequented direction of the total degree homotopy. Namely, for any generic system with a fixed set of degrees, its solutions are implicitly labeled by their B\'ezout index. This is a value which may be computed on each individual solution without needing access to the others. Consequently, the B\'ezout indices function as a natural hash function for monodromy solving using monodromy coordinates \cite{monodromyCoordinates}.

 \subsection{Finding a single solution with certain properties}

   In some applications, one only wants to find a single solution with a given property, e.g., a real solution. 
 In this case, the iterator setup provides not only a low-memory solution, but also a significant reduction in computation time since the computation can terminate once a single solution is found. Indeed, if there are $N$ start solutions and $r$ target solutions have the desired property, then the probability that one path has the desired property is $N/r$ and one expects to find a target solution with the desired property after $r/N$ paths.
  
 We illustrate this with the cyclic $n$-roots problem, which asks for the isolated roots of the  system: 
 \begin{align}\label{eq:ncyclic}
 \textrm{Cyclic}(n)=	\begin{cases}
 		x_0 + x_1 + \cdots + x_{n-1} = 0\\
 		\sum\limits_{i = 0}^{n-1} x_{i} x_{(i+1\!\!\!\!\mod n)} \cdots x_{(i+j \!\!\!\!\mod n)} = 0 \quad \textrm{for $j = 1, \ldots, {n-2}$}. \\
 		x_0 x_1 \cdots x_{n-1} = 1. 
 	\end{cases}
 \end{align}
This is a benchmark problem in numerical algebraic geometry. Let us first implement it in \code{HomotopyContinuation.jl} \cite{HC}.
\begin{lstlisting}[language = code]
using HomotopyContinuation
function cyclic(n)
    @var z[1:n]
    eqs = [sum(prod(z[(k-1)%n+1] for k = j:j+m) for j = 1:n) for m = 0:(n-2)]
    push!(eqs, prod(z) - 1)
    System(eqs, z)
end
\end{lstlisting}
 We use our iterator to find a single solution satisfying the condition that it is real. Here is example code for $n=5$. 
\begin{lstlisting}[language=code]
F = cyclic(5)
I = solve(F, iterator_only = true)
J = Iterators.filter(s -> is_real(s) && is_success(s), I)
first(J)
\end{lstlisting}
The runtime and memory allocations for this experiment may be found in \autoref{fig:ncyclic-one-real}. 
 
 \begin{figure}[!htpb]
 \includegraphics[scale=0.35]{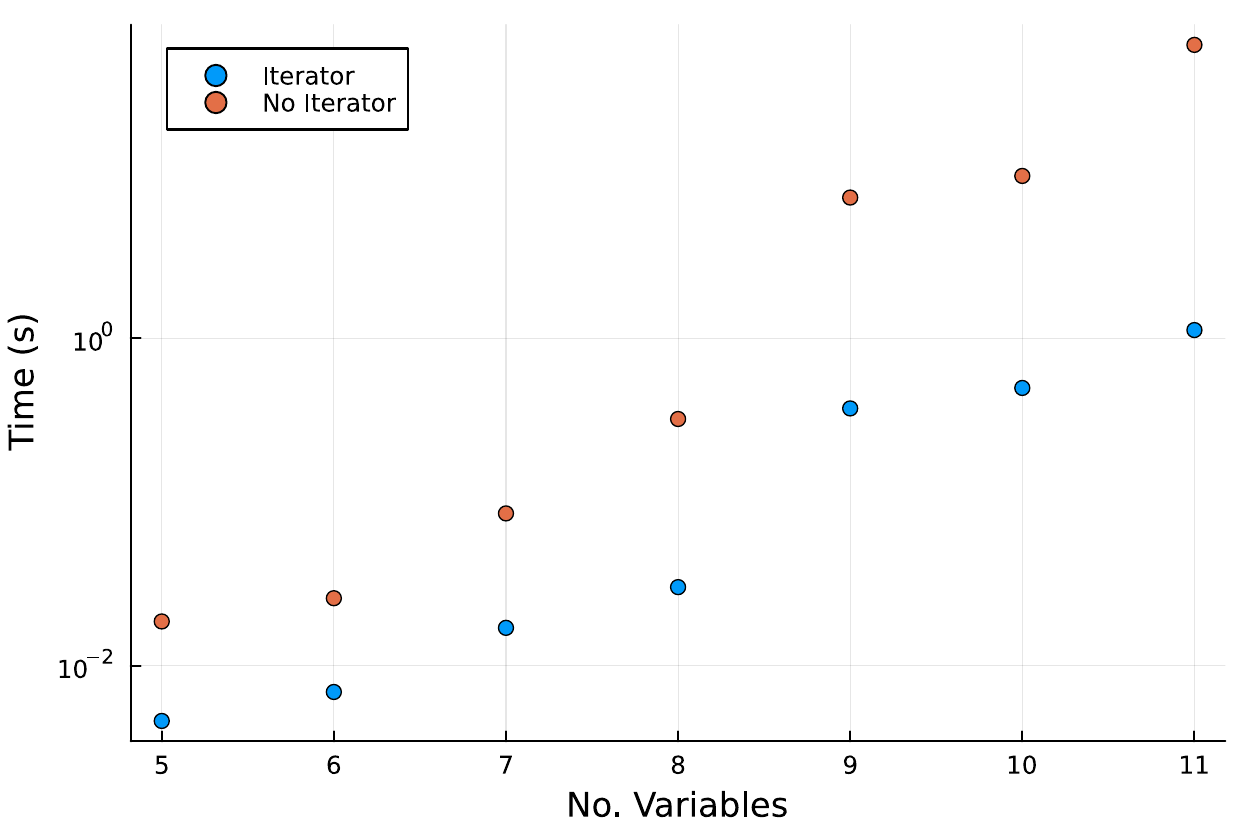}
 \includegraphics[scale=0.35]{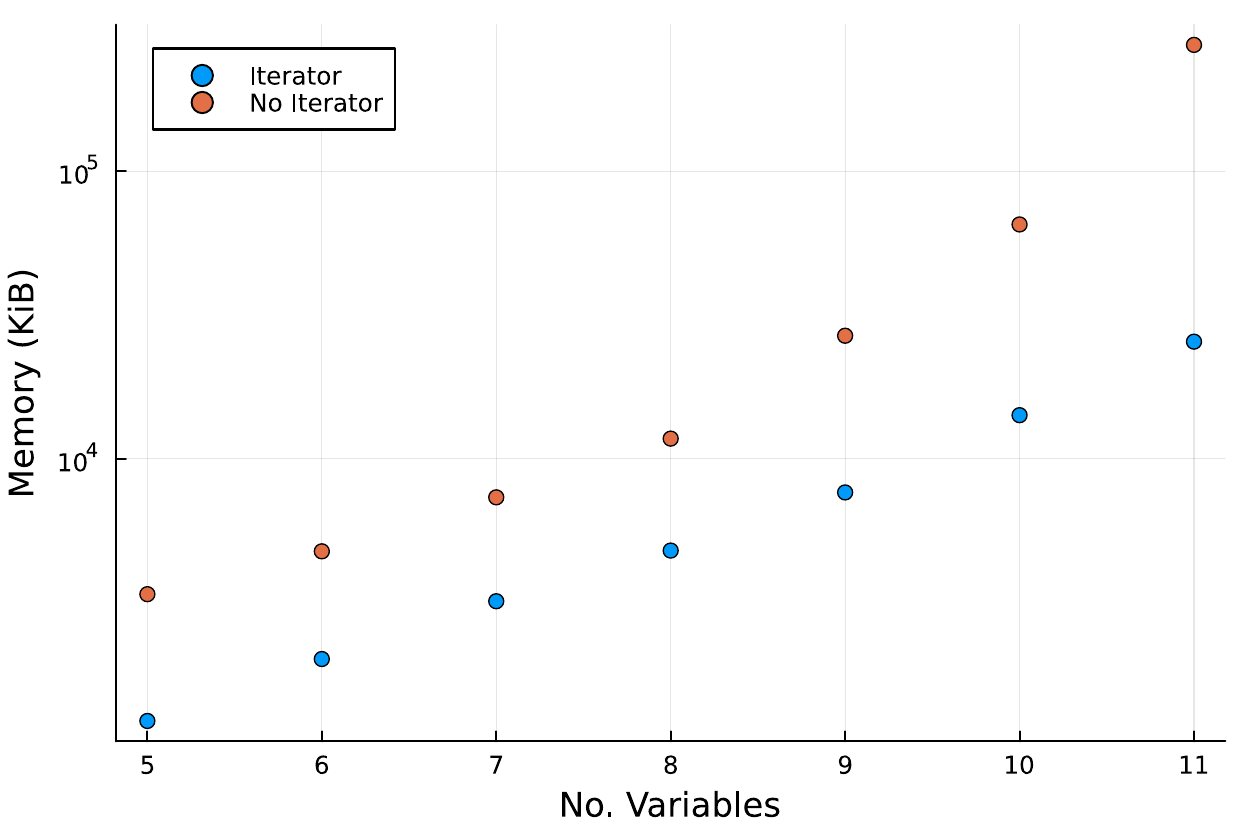}
 \caption{Runtime (left) and memory allocations (right) for the task of finding one real root of \eqref{eq:ncyclic}.}\label{fig:ncyclic-one-real}
 \end{figure}

 \subsection{Brute force sampling for all real solution sets}
Given an enumerative problem $F(\x;\p)$ in variables and parameters, one may seek a parameter value for which the solutions of $F$ exhibit certain behavior. A common example is to find a system with all real solutions. Doing so is not so easy since often the regions in the parameter space for which all solutions are real are extremely small, if they exist at~all.

A brute-force approach to this problem is to repeatedly sample parameter values $\mathbf p$, solve the system $F$, and check if all solutions are real. When $F$ has $d \gg 0$ solutions and we sample $N$ times, this process costs $d\cdot N$ path tracking operations. However, an iterator can recognize, prior 
to its full collection, whether all solutions are real. Here is example code.
\begin{lstlisting}[language=code]
I = solve(F, iterator_only = true)
Iterators.any(s -> is_success(s) && !is_real(s), I)
\end{lstlisting}
The second line will return \code{true} precisely when there is a non-real solution in the solution set of~$F$. The evaluation is lazy: when the first non-real solution is encountered in the iterator, \code{true} is returned without computing the remaining solutions. If there are $k$ non-real solutions out of $n$, then the expected index of the first non-real solution is~$(n+1)/(k+1)$. This drastically reduces the number of paths required for a brute-force search for total reality.

\begin{table}
\begin{tabular}{|c|c|c|c|c|}
 \hline 
\# real & 3 & 7 & 15 & 27 \\ \hline \hline 
Frequency &745530 &210801 &42752 & 917 \\ \hline 
Expected tracks & 1.12 & 1.33 & 2.15 & 27 \\ \hline
\end{tabular}
\bigskip 
\caption{The frequency per million of instances of the problem of $27$ lines with $n-k \in \{3,7,15,27\}$ real solutions along with the expected number of path-tracking operations required to evaluate whether all solutions are real. }
\label{table:27linesdata}
\end{table}

For our formulation of the problem of $27$ lines on a cubic, \autoref{table:27linesdata} gives the frequencies of finding $n-k$ real solutions, out of a sample size of $N=1,000,000$. In particular, it approximates the probability of choosing real parameters for which the problem of $27$ lines has $27$ real solutions by approximately $0.000917$. Thus, the expected number of parameters needed to solve for in order to find such an instance is approximately $\frac{1}{0.000917} \approx 1091$. In the classical setting where a user would solve for all $27$ lines in each instance, this would require a total of $27 \cdot  1091= 29457$ path-tracking operations. However, using homotopy iterators, one expects to use 
\[
\frac{1.12 \cdot 745530 + 1.33 \cdot 210801+ 2.15 \cdot 42752 + 27 \cdot 917}{1000000}= 1.23203473
\]
path-tracking operations on each trial, and thus $ 1091 \cdot 1.23203473 \approx 1344$ many in total. This provides a savings of factor of approximately $22$.

 \subsection{A combinatorial polyhedral start solution iterator}
 \label{secsec:chemistry}
 A homotopy iterator starting at a polyhedral start system (\autoref{sec:polyhedral_start_system}) becomes even more powerful when the mixed cells can be combinatorially iterated over. This is the case in the following example.
  
Let $I_j$ denote the  unit interval ${\rm conv}({\bf 0}, {\bf e}_j)\subset \mathbb{R}^n$ and consider the  sparse polynomial system $F=(f_1,\ldots,f_n)$ with Newton polytopes $
(C_{n,1},\ldots,C_{n,n})$ given by the stretched cubes
\[
C_{n,i} =  I_1 + \cdots + I_{i-1} + 2I_i +I_{i+1}+ \cdots + I_n. 
\]
We suppose each $f_i$ is supported on the lattice points $\mathcal C_{n,i}$ of $C_{n,i}$. For instance, when $n=2$ 
$$
 \begin{aligned}
   f_1(x_1, x_2) &= a_1 + a_2x_1 + a_3x_2 + a_4 x_1 x_2 + a_5 x_1^2 + a_6 x_1^2 x_2 = 0 \\
   f_2(x_1, x_2) &= b_1 + b_2x_1 + b_3x_2 + b_4 x_1 x_2 + b_5x_2^2 + b_6 x_1 x_2^2 = 0.
 \end{aligned}
$$
 We first identify the BKK bound of the system $F(\x)=\textbf{0}$ as a sum over the \mydef{symmetric group} $S_n$ of degree $n$. The formula involves the number $\textrm{Fix}(\sigma)$ of fixed points of a permutation $\sigma$. 
 \begin{lemma}\label{lem:CCelimBKK}
The mixed volume of $(C_{n,1},\ldots,C_{n,n})$ is
 $${\rm MixedVolume}(C_{n,1}, \ldots, C_{n,n}) = \sum_{\sigma \in S_n} 2^{{\rm Fix}(\sigma)}.$$
\end{lemma}
\begin{proof}
  Since
  $$C_{n,i} = I_1 + \cdots + I_{i-1} + 2I_i +I_{i+1}+ \cdots + I_n,$$
by  multilinearity of the mixed volume, we have that  
  \begin{align*}
    {\rm MixedVolume}(C_{n, 1}, \ldots, C_{n, n})
    &=\!\!\! \sum_{\phi \colon [n] \to [n]}\!\!
    {\rm MixedVolume}(2^{\delta_{1, \phi(1)}}I_{\phi(1)}, \ldots, 2^{\delta_{n, \phi(n)}}I_{\phi(n)})\\
    &= \!\!\!\sum_{\phi \colon [n] \to [n]} \!\!2^{{\rm Fix}(\phi)}
    {\rm MixedVolume}(I_{\phi(1)}, \ldots, I_{\phi(n)}).
  \end{align*}
  where $\delta_{i,j}$ is the Kronecker symbol. Although the summation is over \emph{all} functions $\phi \colon [n] \to [n]$, only the terms involving permutations are nonzero. Indeed, if $\phi$ is not bijective then the polytope $\lambda_1 I_{\phi(1)} + \cdots + \lambda_n I_{\phi(n)}$ is not full-dimensional and so the corresponding mixed volume is zero. Thus, the above sum may be taken over permutations $\sigma \in S_d$ and the mixed volume terms are equal by symmetry:
  \begin{align*}
    {\rm MixedVolume}(C_{n, 1}, \ldots, C_{n, n}) &= \sum_{\sigma \in S_n} 2^{{\rm Fix}(\sigma)} {\rm MixedVolume}(I_1, \ldots, I_n). 
  \end{align*}
  Since ${\rm vol}(\lambda_1 I_1 + \cdots + \lambda_n I_n) = \lambda_1 \cdots \lambda_n$, we have that ${\rm MixedVolume}(I_1, \ldots, I_n) = 1$, completing the proof.  
\end{proof}
The sequence of mixed volumes given by \autoref{lem:CCelimBKK} is \href{https://oeis.org/A000522}{A000522} in the OEIS \cite{OEIS:A000522}. The following table shows its growth compared to the B\'ezout bound $(n+1)^n$:

\[
\begin{array}{|c|c|c|c|c|c|c|c|c|c|c|c|c|c|c|}
  \hline
  n & 2 & 3 & 4 & 5 & 6 & 7 & 8 & 9 & 10 & 11\\
  \hline
  \textrm{B\'ezout Bound} & 9 & 64 & 625 & 7776 & 117\,649 & 2\, {\rm mil} & 43\,{\rm mil} & 1\,{\rm bil} & 25\,{\rm bil} & 743\, {\rm bil}\\
  \hline
  \textrm{BKK Bound} & 5 & 16 & 65 & 326 & 1957 & 13\,700 & 109\,601 & 986\,410 & 9\,864\,101 & 108\,505\,112\\
  \hline
\end{array}\]
\smallskip

Although the polyhedral start system is significantly more efficient than the total degree start system, setting it up is costly: for $n=10$ the mixed volume computation in \code{HC.jl} \cite{HC}  takes approximately $45$ minutes.
To avoid this, we bypass the automated construction of a polyhedral start system in~\code{HC.jl}~\cite{HC} by hard-coding its construction explicitly ourselves.

We define the components of the lift $\omega_i \colon \mathcal C_{n,i} \to \mathbb Z$ by 
\begin{equation}\label{weight}\omega_i(v) = i \sum_{j=1}^n v_j,\qquad v \in \mathcal C_{n, i}.\end{equation}
The following ingredients are required to define a \code{MixedCell} in \code{MixedSubdivisions.jl} \cite{MixedSubdivisions}
\begin{itemize}
\item a choice of edge for each polytope $C_{n, j}$,
\smallskip
\item the normal vector $\Tilde{\sigma}$ of the mixed cell,
\smallskip
\item the vector $\beta \in \R^n$ whose $j$th entry is $\min_{v \in C_{n,j}} \langle v,  \Tilde \sigma \rangle$,
\smallskip
\item a boolean indicating whether the mixed cell is a fine mixed cell, and
\smallskip
\item the volume of the mixed cell. 
\smallskip
\end{itemize}
The mixed cell is the Minkowski sum of the edges in the first bullet point.
Note that the vector $\beta$ can be computed from the normal vector $\Tilde \sigma$.
Furthermore, all our mixed cells are fine, since we select one edge from each polytope by construction.
The next result describes the edges, normal vector, and volume of each mixed cell in the subdivision induced by $\omega$. 
\begin{proposition}
  Every mixed cell of the subdivision induced by $\omega$ has normal vector $\tilde{\sigma} = (-\sigma\,\, 1)^T$ for $\sigma$ a permutation; the endpoints $u_j, v_j \in \mathbb{R}^n$ of the corresponding edge for the polytope $C_{n,j}$ are
  \begin{align*}
    u_j(i) = 
    \begin{cases}
      2 \quad &\textrm{if $\sigma(i) > j$ and $i = j$}\\      
      1 \quad &\textrm{if $\sigma(i) > j$ and $i \neq j$}\\
      0 \quad &\textrm{if $\sigma(i) \leq j$}
    \end{cases} && \textrm{and} &&
    v_j(i) = 
    \begin{cases}
      2 \quad &\textrm{if $\sigma(i) \geq j$ and $i = j$}\\
      1 \quad &\textrm{if $\sigma(i) \geq j$ and $i\neq j$}\\
      0 \quad &\textrm{if $\sigma(i) < j$}.
    \end{cases}
  \end{align*}
  The volume of this mixed cell is $2^{\textrm{Fix}(\sigma)}$ where $\textrm{Fix}(\sigma)$ is the number of fixed points of $\sigma$.
\end{proposition}

\begin{proof}
  Given a vector $v \in \mathbb{R}^n$, let $\Tilde{v} = (v^T \,\, \omega(v))^T$ denote the lift of $v$.
  We prove that given a permutation $\sigma \in S_n$, the vector $\Tilde{\sigma} = (-\sigma \,\, 1)^T$ is the normal vector of the Minkowski sum of edges $F_\sigma = (\Tilde{u}_1, \Tilde{v}_1) + \cdots + (\Tilde{u}_n, \Tilde{v}_n)$. 
  A point in $F_\sigma$ has the form $p(\alpha) = \sum_{j = 1}^n \alpha_j \Tilde{u}_j + (1 - \alpha_j) \Tilde{v}_j$ where $\alpha \in [0,1]^{n+1}$.
  The inner product of $p(\alpha)$ with $\tilde \sigma$ is $\langle p(\alpha), \Tilde{\sigma} \rangle =     \sum_{j = 1}^n \alpha_j \langle \Tilde{\sigma}, \Tilde{u}_j \rangle + (1 - \alpha_j)    \langle \Tilde{\sigma}, \Tilde{v}_j \rangle.$
  Since the inner products
  \begin{align*}
    \langle \Tilde{\sigma}, \Tilde{u}_j \rangle \! = \!\!\!\!
    \sum_{\sigma(i) > j} \!\! (j - \sigma(i)) +
    \Big \{j - \sigma(j) \,\, \textrm{if $\sigma(j) > j$}\Big \}\!
    =\!\!\!\!
    \sum_{\sigma(i) > j}\!\! (j - \sigma(i)) + 
    \Big \{j - \sigma(j) \,\, \textrm{if $\sigma(j) \geq j$}\Big \}    \!
    =\! \langle \Tilde{\sigma}, \Tilde{v}_j \rangle 
  \end{align*}
  are equal, the inner product $\langle p(\alpha), \tilde{\sigma} \rangle =     \sum_{j = 1}^n \alpha_j \langle \Tilde{\sigma}, \Tilde{u}_j \rangle + (1 - \alpha_j)    \langle \Tilde{\sigma}, \Tilde{v}_j \rangle = \sum_{j = 1}^n \langle \Tilde{\sigma}, \Tilde{u}_j \rangle$
  does not depend on $\alpha$.
  Therefore the linear form $\langle -, \Tilde{\sigma}\rangle$ is constant on $F_\sigma$ and is therefore a normal vector of $F_\sigma$.
  Furthermore,  $\langle -, \tilde \sigma \rangle$ is minimized on $F_\sigma$, because 
  $\langle p(\alpha), \tilde \sigma \rangle$ is negative and hence less than $\langle {\bf 0}, \tilde \sigma \rangle = 0$.
  Thus $\tilde \sigma$ is an inner normal vector. 
  Because the  last coordinate of $\Tilde{\sigma}$ is positive, the facet $F_\sigma$ is in the lower hull and therefore its projection is a mixed cell.

  The volume of this mixed cell is the product of the lengths of the edges $(u_1, v_1), \ldots, (u_n, v_n)$.
  An edge $(u_j, v_j)$ has length $2$ precisely when $\sigma(j) = j$ and length $1$ otherwise, so this mixed cell has volume $2^{{\rm Fix}(\sigma)}$.
  We have described $n!$ mixed cells indexed by permutations in $S_n$.
  Since the sum of the volumes of these mixed cells equals the mixed volume, these are all the mixed cells. 
\end{proof}
This explicit description of the mixed cells, paired with the ability to iterate over the symmetric group which indexes them, affords us the ability to create a polyhedral start solution iterator at a dramatically reduced cost; see \autoref{fig:hard-coded-iter}. 
\bigskip

\begin{figure}[!htpb]
  \includegraphics[scale=0.35]{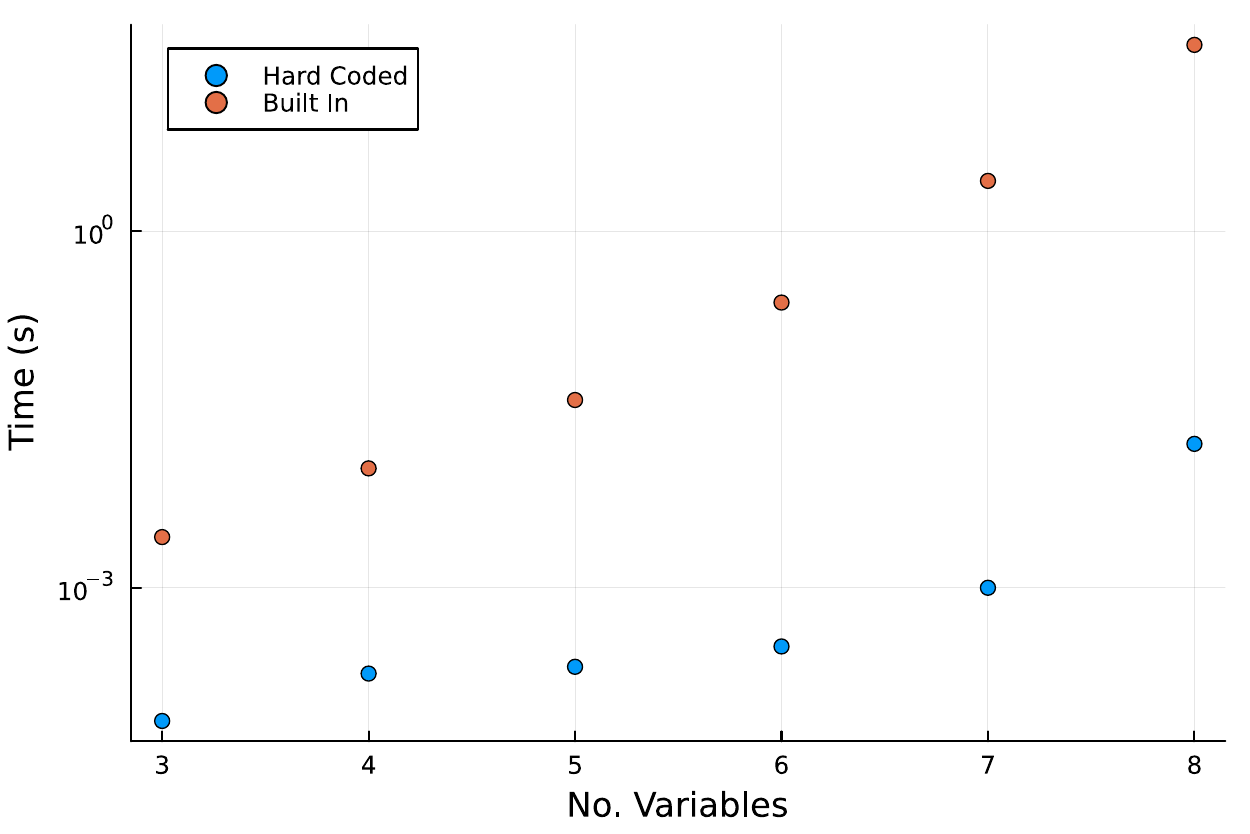}
  \includegraphics[scale=0.35]{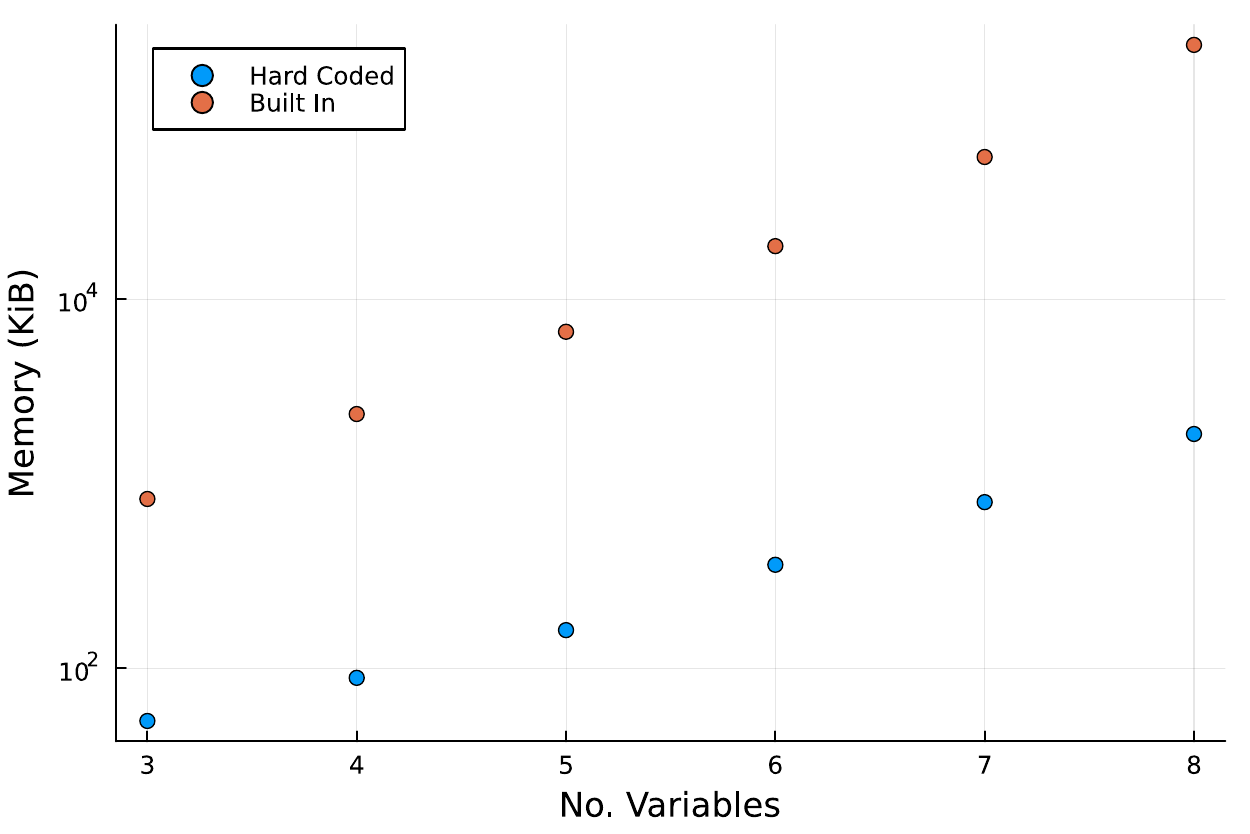}
  \caption{Runtime (left) and memory allocations (right) of setting up a \code{ResultIterator} for a system with random coefficients and support $C_{n,1}, \ldots, C_{n,n}$.}\label{fig:hard-coded-iter}
\end{figure}

We now show how to implement in \code{Julia} \cite{julia} the polyhedral start solution iterator presented in \autoref{sec:polyhedral_start_system} for the example in this section.
First, since we need to call multiple packages, we declare abbreviations for them.
\begin{lstlisting}[language = code]
using Combinatorics, LinearAlgebra, HomotopyContinuation, MixedSubdivisions
const LA = LinearAlgebra; const HC = HomotopyContinuation
const CB = Combinatorics; const MS = MixedSubdivisions
\end{lstlisting}

First, we define a function that computes the weights from \eqref{weight}. \code{HC.jl} uses \code{Int32} numbers for encoding weight vectors.
\begin{lstlisting}[language = code]
stretched_cube(n, i) = map(powerset(1:n)) do s 
   out = [convert(Int32, j in s) for j in 1:n] 
   out[i] = 2*out[i] 
   out
end
stretched_cubes(n) = map(i -> stretched_cube(n,i), 1:n)
weight_vector(n, i) = i .* map(v -> sum(v), stretched_cube(n, i))
weight_vectors(n) = map(i -> Int32.(weight_vector(n, i)), 1:n)
\end{lstlisting}

Next, we define a function that maps a permutation $\sigma$ into the corresponding mixed cell. 

\begin{lstlisting}[language = code]
function perm_to_segments(sigma, n)
    map(1:n) do j
        v1 = [Int(sigma[i] > j) for i in 1:n]; 
        v2 = [Int(sigma[i] >= j) for i in 1:n]
        v1[j] = 2*v1[j]; v2[j] = 2*v2[j]
        (v1, v2)
    end
end
function perm_to_mixedcell(sigma, cubes, n)
    nfix = count(i -> sigma[i] == i, 1:n)
    segm = perm_to_segments(sigma, n)
    indices = [(findfirst(==(segm[i][1]), cubes[i]), 
                findfirst(==(segm[i][2]), cubes[i])) for i in 1:n]
    augmented_segments = [([segm[i][1]; sum(segm[i][1]) * i], 
                            [segm[i][2]; sum(segm[i][2]) * i]) for i in 1:n]
    beta = [minimum([-sigma; 1]' * hcat(s...)) for s in augmented_segments]
    MS.MixedCell(indices, -sigma, beta, true, 2^nfix)   
end
\end{lstlisting}
Finally, we define an iterator for the mixed cells.
\begin{lstlisting}[language = code]
function mixed_cell_iterator(n)
    cubes = stretched_cubes(n)
    perms = CB.permutations(1:n)
    Iterators.map(sigma -> perm_to_mixedcell(sigma, cubes, n), perms)
end
\end{lstlisting}

Now that we have an iterator for the mixed cells, we are ready to implement a start solution iterator. \code{HC.jl} \cite{HC} provides an object called \code{PolyhedralStartSolutionsIterator}, which needs as input the support of the polynomial system, coefficients of a generic system with that support, the lifting and an iterator for mixed cells. We already defined the last two, so let us now implement the others. Our example code is for $n=5$. First we define the mixed cell iterator, the support of our problem and the lifting. 
\begin{lstlisting}[language = code]
n = 5
cells = mixed_cell_iterator(n)
support = map(s -> hcat(s...), stretched_cubes(n))
lifting = weight_vectors(n)
\end{lstlisting}
Suppose moreover, that the system we want to solve has coefficients \code{target_coeffs}. For instance, random real coefficients of the right size can be defined as follows.
\begin{lstlisting}[language = code]
target_coeffs = [randn(Float64, 2^n) for _ in 1:n]
\end{lstlisting}
Then, we define a generic system with the given support, and sample generic parameters \code{gen_coeffs} (to improve numerical stability, generic coefficients should be generated with the same magnitude as~\code{target_coeffs}. For clarity of exposition, we sample only random Gaussian numbers here),
\begin{lstlisting}[language = code]
@var x[1:n]
gen_coeffs = [randn(ComplexF64, 2^n) for _ in 1:n]
F = fixed(HC.polyhedral_system(support))
\end{lstlisting}
We are ready to define the \code{PolyhedralStartSolutionsIterator}:
\begin{lstlisting}[language = code]
iter = HC.PolyhedralStartSolutionsIterator(support, gen_coeffs, lifting, cells)
\end{lstlisting}
The last ingredient for solving is a \code{PolyhedralTracker} that declares the homotopy being used to push the start solutions in \code{PolyhedralStartSolutionsIterator} forward. In \code{HC.jl} \cite{HC} this is defined through a \code{ToricHomotopy} that tracks from the binomial start system to the generic system defined by \code{gen_coeffs}, and by a \code{CoefficientHomotopy} that tracks from that generic system to our target system defined by \code{target_coeffs}. All this information is saved in a \code{Solver} at the end:
\begin{lstlisting}[language = code]
H1 = HC.ToricHomotopy(F, gen_coeffs)
toric_tracker = Tracker(H1)
H2 = begin
    p = reduce(append!, gen_coeffs; init = ComplexF64[])
    q = reduce(append!, target_coeffs; init = ComplexF64[])
    HC.CoefficientHomotopy(F, p, q)
end;
gen_tracker = EndgameTracker(Tracker(H2))
tracker = HC.PolyhedralTracker(toric_tracker, gen_tracker, support, lifting)
S = Solver(tracker; start_system = :polyhedral)
\end{lstlisting}
This culminates in the definition of a \code{ResultIterator} encoding the corresponding homotopy iterator.
\begin{lstlisting}[language = code]
I = ResultIterator(iter, S)
\end{lstlisting}
As before, this iterator can be collected or otherwise manipulated. The fundamental strength of this construction is that the explicit definition of the mixed cell iterator let us set up the homotopy iterator without computing the start system explicitly.

\bibliographystyle{abbrvnat}
{\small
\bibliography{LowMemoryNumerics}}
\end{document}